\documentclass[12pt,final]{amsart}

\usepackage{amssymb}

\usepackage{enumerate}

\usepackage{graphicx}

\makeatletter
\@namedef{subjclassname@2010}{%
  \textup{2010} Mathematics Subject Classification}
\makeatother

\newtheorem{thm}{Theorem}[section]
\newtheorem{cor}[thm]{Corollary}
\newtheorem{lem}[thm]{Lemma}
\newtheorem{prop}[thm]{Proposition}

\theoremstyle{definition}

\newtheorem{as}{Assumption}
\numberwithin{equation}{section}

\newcommand{\et}{\mathbb{E}_{x,t}}

\newcommand{\Tr}{\operatorname{Tr}}

\begin{document}

%%%%% To ease editing, for IMPAN journals add:

\baselineskip=15pt

%%%%%%%%%%%

%% In the running head, replace first names by initials 
%% and give an abbreviation of the title.

\title[Existence results  to Isaacs equations]
 {Existence results for Isaacs equations with local conditions and related semilinear Cauchy problems}

%----------Author 1
\author[D. Z.]{Dariusz Zawisza \\{\tiny Institute of Mathematics, Jagiellonian University in Krakow,\\ dariusz.zawisza@im.uj.edu.pl}}

\address{\indent Institute of Mathematics \newline \indent Faculty of Mathematics and Computer Science \newline \indent  Jagiellonian University in Krakow \newline \indent{\L}ojasiewicza  6 \newline \indent 30-348 Krak{\'o}w, Poland }

\email{dariusz.zawisza@im.uj.edu.pl}

\date{}

%\thanks{This work was completed with the support of our
%\TeX-pert.}
%%----------Author 2
%\author{A Second Author}
%\address{The address of\br
%the second author\br
%sitting somewhere\br
%in the world}
%\email{dont@know.who.knows}
%%----------classification, keywords, date
\subjclass[2010]{35K58,49J20,91A15,91A23}
\keywords{Cauchy problem, Hamilton Jacobi Bellman Isaacs equation, robust control, semilinear parabolic equation, stochastic game}

\begin{abstract}
Our goal is to prove existence results for classical solutions to some general nondegenerate Cauchy problems which are  natural generalizations of Isaacs equations. For the latter we are able to extend our results by admitting local conditions for  coefficients. Such equations appear naturally for instance in  robust control theory. Using our general results, we can solve not only Isaacs equations, but also equations for other sophisticated  control problems, for instance models with state dependent constraints on the control set. 
%In the proof we use a fixed point type argument, with an operator which takes advantage of fundamental solutions to linear parabolic equations. We further extend it to include some local conditions by making some approximations and use the Arzell-Ascolli Lemma. 
\end{abstract}

%%% ----------------------------------------------------------------------
\maketitle

\begin{center}{\small
First version: 31 August 2016 \\
This version: published in Ann. Polon. Math. 121 (2018), 175--196}
\end{center}

\section{Introduction}

Our main concern here is to prove some general results regarding a classical solution ($u \in \mathcal{C}^{2,1}(\mathbb{R}^{N}\times[0,T)) \cap \mathcal{C}(\mathbb{R}^{N}\times[0,T]) $ ) to the semilinear Cauchy problem of the type

\begin{equation} \label{pierwsze}
 \begin{cases}
u_{t}+\tfrac12\Tr(a(x,t)D_{x}^{2}u)+ H(D_{x}u,u,x,t)=0, &\quad (x,t) \in \mathbb{R}^{N} \times [0,T), \\
u(x,T)=\beta(x), &\quad x \in \mathbb{R}^{N}.
\end{cases}
\end{equation}
We use $u_{t}$ to denote the derivative with respect to $t$, $D_{x} u$ to denote the gradient $(u_{x_{1}},u_{x_{2}},\ldots ,u_{x_{N}})$ and $D_{x}^{2} u$ is used to denote the matrix of the second order derivatives.

Our motivation comes from the fact that equation \eqref{pierwsze} can be used as an excellent starting point to solve  many control and dynamic game problems. However, in the existing literature it is usually hard to find sufficiently general and easily verifiable results for classical solutions which can be directly applied to the HJB theory. For instance, equation \eqref{pierwsze} is a natural generalization of the following Isaacs type equation:
\begin{multline}  \label{equationL1} 
 u_{t}+ \tfrac12\Tr( a(x,t)D^{2}_{x} u) \\  +\max_{\delta \in D} \min_{\eta \in \Gamma} \biggl(i(x,t,\delta,\eta) D_{x}u  +  h(x,t,\delta,\eta) u+ f(x,t,\delta,\eta)\biggr) =0, \\ \quad (x,t) \in \mathbb{R}^{N} \times [0,T)
\end{multline}
with the terminal condition
$u(x,T)=0$, where $D \subset \mathbb{R}^{k}$ and $\Gamma  \subset \mathbb{R}^{l}$ are fixed compact sets. In stochastic control context, the existence of a classical solution is often crucial to determine the optimal control/saddle point and helpful to establish a convergence rate for numerical methods. To explore this topic more, it is worth to read Dupuis and James \cite{Dupuis}.

Equation \eqref{equationL1} is very popular in stochastic game theory and has gained a lot of attention recently in robust stochastic optimal control, where it is used to solve optimization problems with  model ambiguity (or model misspecification). For financial aspects of model ambiguity  see for example  Hern\'{a}ndez-Hern\'{a}ndez and Schied \cite{Schied4}, Schied \cite{Schied3}, Tevzadze et al. \cite{tevzadze}, Zawisza \cite{Zawisza2} and references therein. For a discussion concerning robust control in environmental economics see Xepapadeas \cite{x}, Jasso-Fuentes and L\'{o}pez--Barrientos \cite{jasso2} or L\'{o}pez- Barrientos et al. \cite{jaso}. In fact they formulate problems in the infinite time horizon setting, but there is no problem in rewriting it in the fixed time framework. The latter work provides general existence results for  classical solutions to the associated elliptic Isaacs equations. 

Moreover, equation \eqref{equationL1} might be used as well as the first step of solving ergodic control problems: for the risk sensitive optimization see Fleming and McEneaney \cite{mcefle} and  Zawisza \cite{zawisza} for the consumption - investment problem.

Equation \eqref{pierwsze} can be used not only to solve Isaacs equation, but also to other non-standard control problems. In finance, it can be applied to solve recursive utility problems,  for example these considered by Kraft et al. \cite{Kraft}. We would like to lay the emphasis  on the  stochastic control problems with  state dependent bounds for the control set. At the end of the second section we will present some particular optimal dividend problem linked to this issue. 

Apart from stochastic control applications, our paper has some useful applications in other fields. First of all, for the last few decades, many researchers have investigated the theory of parabolic equations with unbounded coefficients. For the recent contribution in this field see Kunze et al. \cite{Kunze}, Angiuli and Lunardi \cite{Angiuli} and the survey paper of Lorenzi \cite{Lorenzi}. Our Theorem \ref{main} provides some new existence results in this area. 

In addition, our work might be helpful in proving the existence results for forward-backward stochastic systems. The detailed analysis is contained in Ma and Yong \cite[Chapter 4]{Ma}. The link between backward equations and quasilinear equations is mutual i.e. some results concerning existence theorems for partial differential equations  can be proved by applications of backward stochastic equations.
One of the most general results concerning existence of solutions to equations \eqref{pierwsze} and \eqref{equationL1} can be deduced from $W^{2,1}$ theorem proved by BSDE methods in Delarue and Guatteri \cite{Delarue}. Their results are strong enough to cover as well our existence results under our Assumption 1 (Theorem \ref{thefirst}). However, the importance of our proof lay in the fact that we use the fixed point method with respect to a norm, which ensures that the solution can be uniformly approximated by the solutions to the linear equations  and guarantees relatively fast convergence together with the first derivative. 

During the peer revision process we have also discovered  that the  same set of conditions (Assumption 1) is largely covered by the recent result of Addona et al. \cite[Theorem 3.6]{Addona} and proved by exploiting the fixed point approach. But, they have used  slightly different technique which operates on the solution defined on the small time interval $(T-\delta,T]$  and they have not proved global uniform convergence to the fixed point. Moreover, they assume $C^{1+\alpha}$ regularity in the space variable for the second order coefficient $a$.  

There are of course some other works. Kruzhkov and Olejnik's \cite{Kruzhkov}  and Friedman's \cite{friedman} results work for many Isaacs equations but with trivial second order term ($a=I$). Rubio \cite{Rubio} considered only stochastic control formulation which is not directly applicable to the max-min framework and other
 semilinear equations mentioned in this paper. In addition, our last result (Theorem \ref{last}) is strong enough to extend Rubio's \cite{Rubio} results to the case when the functions  $f$ and $\beta$ satisfy the exponential growth condition and the function $h$ has the linear growth condition. Ma and Yong's theorem \cite[Chapter 4]{Ma} holds under smoothness conditions which are not precisely indicated. In addition, it is worth mentioning, that standard stochastic control books such as Fleming and Rishel \cite{FlemingRishel} and Fleming and  Soner \cite{Fleming} provide general results, but still they are not sufficient for many applications. We should also recall here the work of Addona \cite{A} where some existence results concerning so called mild solutions to equation \eqref{pierwsze} are considered and Fleming and Souganidis \cite{FS} where the value function of the suitable game is proved to be a viscosity solution to \eqref{equationL1} under a global Lipschitz condition for coefficients. 

Our paper is structured as follows. First, we will prove the existence result under conditions which allows us to apply the approach based on the fundamental solution and fixed point arguments. The fixed point approach can be useful to obtaining numerical solution to our equation. Further, we will extend it to allow some local conditions by making some approximations and transforming the equation into the form which enables us to use the stochastic representation. Such type of approximation was used ealier in Zawisza \cite{zawisza3} to prove an existence result for some infinite horizon control problems.
At the end we will focus on the explicit Isaacs equation for the stochastic game formulation.

\section{General results}
We start with proving the existence theorem under conditions listed in Assumption \ref{as1}. Further, we will apply it to prove a suitable result under conditions given in Assumption \ref{as2}.

\begin{as} \label{as1}
\text{ }
\begin{itemize}
\item[{\bf A1)}]
The matrix $a$ is of the form $a=\sigma \sigma^{T}$, where the coefficients  $\sigma_{i,j}(x,t)$, \\ $i,j=1,2,\ldots,n$, are uniformly bounded, Lipschitz continuous on compact subsets in $\mathbb{R}^{N} \times [0,T]$, and Lipschitz continuous in $x$ uniformly with respect to $t$. In addition there exists a constant $\mu>0$ such that for any  $\xi \in \mathbb{R}^{N} $
\[ 
\sum_{i,j=1}^{N} a_{i,j}(x,t) \xi_i \xi_{j} \geq \mu |\xi|^{2}, \quad (x,t) \in \mathbb{R}^{N} \times [0,T].
 \]
 \item[{\bf A2)}] The function $\beta$ is bounded and uniformly Lipschitz continuous.
 
 \item[{\bf A3)}] The function $H$ is  H\"{o}lder continuous on compact subsets of $\mathbb{R}^{2 N+1}\times[0,T)$. Moreover, let there exist $K>0$ such that for all\\ $(p,u,x,t),\; (\bar{p},\bar{u},x,t) \in \mathbb{R}^{2N+1} \times [0,T]$
\begin{equation}\label{warH1}
\begin{aligned}
&|H(p,u,x,t)| \leq K\bigl(1+|u|+|p|\bigr),\\
&|H(p,u,x,t)-H(\bar{p},\bar{u},x,t)| \leq K \bigl(|u-\bar{u}| + |p - \bar{p}|\bigr). 
\end{aligned}
\end{equation}
\end{itemize} 
\end{as}

Let $C_{b}^{1,0}$ stand for the space of all functions which are continuous,\\ bounded and have the first derivative with respect to $x$ which is also continuous and bounded. The space is equipped with the family of norms:  
\begin{equation} \label{norm}
\|u\|_{\kappa}: = \sup_{(x,t) \in \mathbb{R}^{N} \times (0,T]} e^{-\kappa(T-t)} |u(x,t)| +  \sup_{(x,t) \in \mathbb{R}^{N} \times (0,T)} e^{-\kappa(T-t)} |D _{x} u(x,t)|.
\end{equation}
Note that the space $C_{b}^{1,0}$ together with $\|\cdot\|_{\kappa}$ forms a Banach space. This norm was inspired by the work of Becherer and Schweizer \cite{Becherer}. They use this definition of the norm, but without the gradient term. In their paper they have solved some semilinear equations, but their setting  excludes the nonlinearity in the gradient part. The norm \eqref{norm} has also been used by Berdjane and Pergamentschikov  \cite{Berdjane} to solve semilinear equations in the consumption investment problem, but still the nonlinearity in the equation involves only the zero order term $u$.

If we consider first the linear equation 
\begin{equation*}
 \begin{cases}
  u_{t}+ \tfrac12\Tr( a(x,t)D_{x}^{2}u) + f(x,t)=0,  &\quad (x,t) \in \mathbb{R}^{N} \times [0,T), \\
u(x,T)=\beta(x), &\quad x \in \mathbb{R}^{N},
\end{cases}
\end{equation*}
then it is well known (see Friedman \cite[Chapter 1, Theorem 12]{F:B})  that under {\bf A1} and {\bf A2}, for $f$ being bounded and locally H\"{o}lder continuous in $x$ uniformly with respect to $t$ on compact subsets of $\mathbb{R}^{n} \times [0,T) $, there exists a unique bounded classical solution which is given by the formula
\[
u(x,t)= \int_{\mathbb{R}^{N}} \beta(y) \Gamma(x,t,y,T) dy + \int_{t}^{T} \int_{\mathbb{R}^{N}} \Gamma(x,t,y,s)f(y,s) dy ds,
\]
where $\Gamma(x,t,y,s)$ is the fundamental solution to the problem
\[
\Gamma_{t}+\frac{1}{2} Tr(a(x,t)D_{x}^{2} \Gamma)=0.
\]
Moreover,
\begin{equation} \label{111}
\int_{\mathbb{R}^{N}} \Gamma(x,t,y,s) dy =1, \quad \text{for} \; x \in \mathbb{R}^{N},\; 0 \leq t <s \leq T,
\end{equation}
functions $\Gamma$,  $\Gamma_{t}$, $D_{x} \Gamma$,  $D^{2}_{x} \Gamma$ are continuous on the set $x,y \in \mathbb{R}^{N}$, $0 \leq t <s \leq T$, and there exist $c,C>0$ such that 
\begin{align} 
|\Gamma(x,t,y,s)| &\leq \frac{C}{(s-t)^{N/2}} \exp \biggl(-c \frac{|y-x|^{2}}{s-t}\biggr), \notag\\ |D_{x} \Gamma(x,t,y,s)| &\leq \frac{C}{(s-t)^{(N+1)/2}} \exp \biggl(-c \frac{|y-x|^{2}}{s-t}\biggr), \label{fundamental:in}
\end{align}
(see Friedman \cite[ Chapter 6, Theorem 4.5]{friedman2}). In fact Theorem 12 in Friedman  \cite{F:B} requires that the function $f$ should be H\"older continuous in $x$ uniformly with respect to $t \in [0,T]$. Nonetheless, for uniformity restricted to compact subsets of $[0,T)$ the result can be proved in the same way, because for $t < T_{0}<T$ we can write 
\begin{multline*}
\int_{t}^{T} \int_{\mathbb{R}^{N}} \Gamma(x,t,y,s)f(y,s) dy ds \\ = \int_{t}^{T_{0}} \int_{\mathbb{R}^{N}} \Gamma(x,t,y,s)f(y,s) dy ds + \int_{T_{0}}^{T} \int_{\mathbb{R}^{N}} \Gamma(x,t,y,s)f(y,s) dy ds.
\end{multline*}
The first integral $\int_{t}^{T_{0}} \int_{\mathbb{R}^{N}} \Gamma(x,t,y,s)f(y,s) dy ds$ can be treated as in the Friedman's proof. In the second one, there is no singularity and standard theorems about differentiation under the integral sign can be applied.

We consider as well the subspace $C_{b,h}^{1,0}$ consisting of all functions $u$ such that:
\begin{enumerate}
\item $u \in C_{b}^{1,0}$,
\item for any pair of compact sets $B \subset \mathbb{R}^{n}$, $U \subset (0,T)$ there exist $L>0$ and $\gamma \in (0,1]$ such  that
\[|D_{x}u(x,t) - D_{x}u(\bar{x},t)| \leq L|x-\bar{x}|^{\gamma}, \quad (x,t), (\bar{x},t) \in B \times U.
\]
\end{enumerate}

Note that the subspace $C_{b,h}^{1,0}$ might not be closed in $\|\cdot\|_{\kappa}$ and therefore it is not generally a Banach space.
We can define the mapping 
\begin{multline}
\mathcal{T} u (x,t): =  \int_{\mathbb{R}^{N}} \beta(y) \Gamma(x,t,y,T) dy  \\ + \int_{t}^{T}\int_{\mathbb{R}^{N}} H(D_{x} u(y,s),u(y,s),y,s) \Gamma(x,t,y,s) dy ds.
\end{multline}

\begin{prop}
Under Assumption \ref{as1}  the operator $\mathcal{T}$ maps $C_{b,h}^{1,0}$ into $C_{b,h}^{1,0}$ and there exists $\kappa>0$ such that the operator $\mathcal{T}$ is a contraction with respect to $\|\cdot\|_{\kappa}$.
\end{prop}
\begin{proof}
Suppose that the function $f$ is continuous, bounded and locally H\"{o}lder continuous in $x$ uniformly wrt. $t \in U$, for any compact set $U \subset (0,T)$. We consider first two functions 
\begin{align*}
v_{1}(x,t):&=\int_{\mathbb{R}^{N}} \beta(y) \Gamma(x,t,y,T) dy, \\ v_{2}(x,t):&=\int_{t}^{T} \int_{\mathbb{R}^{N}} \Gamma(x,t,y,s)f(y,s) dy ds.
\end{align*}
Both are bounded and continuous.
Note that from Feynman-Kac formula
\[
v_{1}(x,t)=\et \beta(X_{T}), \quad v_{2}(x,t):= \et \int_{t}^{T} f(X_{s},s) ds,
\]
where $dX_{t} =\sigma(X_{t})dW_{t}$, $\sigma \sigma^{T}=a$ and $\et$ stands for the expected value when the system starts at $(x,t)$.
Standard estimates for diffusion processes (see Friedman  \cite[Chapter 5,Lemma 3.3]{friedman2}) ensure that $v_{1}(x,t)$ is globally Lipschitz continuous in $x$ uniformly with respect to $t$. For
 the function $v_{2}$ we have
\[
D_{x} v_{2}(x,t)= \int_{t}^{T} \int_{\mathbb{R}^{N}} D_{x} \Gamma(x,t,y,s)f(y,s)dy ds,
\]
(see Friedman \cite[Chapter 1, Theorem 3]{F:B}).
From \eqref{fundamental:in} and 
\begin{equation} \label{sugestia}
\int_{\mathbb{R}^{N}}  \left[\frac{c}{ 4 \pi (s-t)} \right]^{\frac{N}{2}} \exp \left[ - \frac{c |x-y|^{2}}{(s-t)}\right] dy =1, \quad s>t, x \in \mathbb{R}^{N},
\end{equation}
we get
\begin{align} \label{festimate}
|D_{x} v_{2}(x,t)| &\leq \int_{t}^{T} \int_{\mathbb{R}^{N}}  |\frac{C}{(s-t)^{(N+1)/2}} \exp \biggl(-c \frac{|x-y|^{2}}{s-t}\biggr)f(y,s)dy ds\\ &\leq C \left[ \frac{4 \pi}{c}\right]^{\frac{N}{2}}\|f\| \int_{t}^{T}\frac{1}{\sqrt{s-t}} ds =2 C \left[ \frac{4 \pi}{c}\right]^{\frac{N}{2}}\|f\| \sqrt{T-t}, \notag
\end{align}
where $\|f\|$ stands for the classical $sup$ norm of the function $f$.
For $u \in C_{b,h}^{1,0}$ we can set $f(x,t):=H(D_{x} u(x,t),u(x,t),x,t)$. We already know that $w:=\mathcal{T} u$ is a classical solution to 
\[
  w_{t}+ \frac{1}{2} Tr( a(x,t)D_{x}^{2}w) + H(D_{x} u(x,t),u(x,t),x,t) =0,  \; (x,t) \in \mathbb{R}^{N} \times (0,T)
\]
with the terminal condition $w(x,T)=\beta(x)$. In particular $D_{x} u$ is Lipschitz continuous on compact subsets of $\mathbb{R}^{N} \times (0,T)$. This fact together with inequality \eqref{festimate} ensure 
that the operator $\mathcal{T}$ maps $C_{b,h}^{1,0}$ into $C_{b,h}^{1,0}$.  
Next two estimates will show that $\mathcal{T}$ is a contraction for  sufficiently large $\kappa>0$. Using 
inequality $\eqref{warH1}$, property $\eqref{111}$ and 
\[
e^{-\kappa (T-t)}\int_{t}^{T} e^{\kappa(T-s)} ds = \frac{1}{\kappa} e^{-\kappa (T-t)} \left[e^{\kappa(T-t)}-1 \right] \leq \frac{1}{\kappa},
\]
 we get
\begin{align*}
&e^{- \kappa (T-t)}\bigl|\mathcal{T}u(x,t)-\mathcal{T}v(x,t)\bigr| \\ & \begin{aligned} \leq    e^{- \kappa (T-t)} K \int_{t}^{T} \int_{\mathbb{R}^{N}} (|u(y,s) -v(y,s)| + |D_{x} u(y,s) - &D_{x} v(y,s)|) \\ & \times \Gamma(x,t,y,s) dy  ds \end{aligned}\\  &\leq e^{- \kappa (T-t)} K  \|u-v\|_{\kappa} \int_{t}^{T}\int_{\mathbb{R}^{N}} \Gamma(x,t,y,s) e^{ \kappa (T-s)}dy ds \leq \frac{K}{\kappa} \|u-v\|_{\kappa}.
\end{align*}
In addition,
\begin{align*}
&e^{- \kappa (T-t)} \bigl|D_{x}(\mathcal{T}u(x,t)-\mathcal{T}v(x,t))\bigr| = \\ & \begin{aligned} \leq  e^{- \kappa (T-t)} K \int_{t}^{T} \int_{\mathbb{R}^{N}} | H(D_{x}u(y,s),u(y,s),y,s)- & H(D_{x}v(y,s),v(y,s),y,s) | \\  & \times |D_{x}\Gamma(x,t,y,s)| dy ds  \end{aligned} \\
& \begin{aligned}\leq e^{- \kappa (T-t)} K \int_{t}^{T}  \int_{\mathbb{R}^{N}}&\bigl(|u(y,s)-v(y,s)|+|
D_{x} u (y,s)- D_{x}v(y,s)|\bigr) \\
& \times \frac{C}{(s-t)^{(N+1)/2}} \exp \biggl(-c \frac{|y-x|^{2}}{s-t}\biggr) dy ds. \end{aligned} 
\end{align*}
Once again, \eqref{sugestia} implies that there exists $\bar{M}>0$ such that 
\begin{align*}
&e^{- \kappa (T-t)} \bigl|D_{x}(\mathcal{T}u(x,t)-\mathcal{T}v(x,t))\bigr|\leq \bar{M} e^{ \kappa t} \|u-v\|_{\kappa} \int_{t}^{T} \frac{e^{-\kappa s}}{\sqrt{s-t}} ds 
\\ 
&\quad \leq \bar{M} e^{ \kappa t} \|u-v\|_{\kappa} \biggl(\int_{t}^{T} (s-t)^{-\frac{3}{4}} ds\biggr)^{\frac{2}{3}} \biggr( \int_{t}^{T} e^{-3\kappa s} ds \biggr)^{\frac{1}{3}}.
\end{align*}
We have
\[ e^{ \kappa t}
\left[\int_{t}^{T} e^{-3\kappa s} ds\right]^{\frac{1}{3}} =e^{ \kappa t}\left[\frac{1}{3 \kappa} \left[e^{-3\kappa t}- e^{-3\kappa T}\right]\right]^{\frac{1}{3}} \leq \frac{1}{\sqrt[3]{3 \kappa}}.
\]
Therefore, there exists a constant $L>0$, dependent only on the time horizon $T$, such that
\[
\sup_{(x,t) \in \mathbb{R}^{N} \times (0,T)} e^{- \kappa (T-t)} \bigl|D_{x}(\mathcal{T}u(x,t)-\mathcal{T}v(x,t))\bigr| \leq \frac{L}{\sqrt[3]{\kappa}} \|u-v\|_{\kappa}.
\]
\end{proof}

\begin{thm} \label{thefirst} Under Assumption \ref{as1}, there exists  a solution  $u \in \mathcal{C}^{2,1}(\mathbb{R}^{N}\times(0,T)) \cap \mathcal{C}(\mathbb{R}^{N}\times(0,T])$ to \eqref{pierwsze} which in addition is bounded together with $D_{x} u$.
\end{thm}
\begin{proof}
Repeating the proof of the Banach Theorem we can take any $u_{1} \in C^{1,0}_{b,h}$ and define the sequence $u_{n+1}=\mathcal{T} u_{n}$, $n \in \mathbb{N}$. Because the operator $\mathcal{T}$ is a contraction in the norm, there exists $\delta \in (0,1)$ such that 
\[
\|u_{n+1}-u_{n}\|_{\kappa} \leq \delta^{n}\|u_{2}-u_{1}\|_{\kappa} ,\quad n\in \mathbb{N}.
\]
Hence,
\[
\|u_{m}-u_{n}\|_{\kappa} \leq \sum_{k=n}^{m-1} \delta^{k}\|u_{2}-u_{1}\|_{\kappa},\quad m>n,
\]
which implies that $u_{n}$ is a Cauchy sequence and consequently it is convergent to $u\in C^{1,0}_{b}$ in the norm $\|\cdot\|_{\kappa}$. The convergence of the norm implies that the sequence $D_{x}u_{n}$ is convergent uniformly to some function $v \in \mathcal{C}(\mathbb{R}^{N} \times [0,T))$. In particular, we have $v=D_{x} u$.  
Moreover, the function $u$ is a fixed point of $\mathcal{T}$.
To complete the reasoning it is sufficient to prove that $u$ belongs also to the class  $C^{1,0}_{b,h}$. Let us first note that the sequence $u_{n}$ is convergent in $\|\cdot\|_{\kappa}$ (for $\kappa$ large enough). Therefore, sequences $u_{n}$ and $D_{x}u_{n}$ are bounded uniformly with respect to $n$. We can now combine (E8), (E9) from Fleming and Rishel \cite[Appendix E]{FlemingRishel} to prove a uniform bound on compact subsets for H\"{o}lder norm of $D_{x} u_{n}$ i. e. for all $k \in \mathbb{N}$ there exist $L_{k} >0$,  $\gamma_{k} \in (0,1]$ such that for all $n \in \mathbb{N}$
\[
|D_{x} u_{n}(x,t) - D_{x} u_{n}(\bar{x},t)|  \leq L_{k} |x-\bar{x}|^{\gamma_{k}}, \quad (x,t),(\bar{x},t) \in B_{k} \times [\delta_{k},t_{k}], 
\]
where $B_{k}=\{x\in \mathbb{R}^{N}|\; |x| \leq k\}$ and $\{\delta_{k}\}_{k \in \mathbb{N}}$ and $\{t_{k}\}_{k \in \mathbb{N}}$ are sequences converging to $0$ and $T$ respectively. Letting $n \to +\infty$ proves  that  $D_{x} u \in  C^{1,0}_{b,h}$. 
\end{proof}

Now we describe the second set of conditions.
\begin{as} \label{as2}
 \text{ }
\begin{itemize}
\item[{\bf B1)}]
The matrix $a_{i,j}(x,t)$ is Lipschitz continuous on compact subsets in $\mathbb{R}^{N} \times [0,T]$. In addition, there exists a constant $\mu>0$ such that for any  $\xi \in \mathbb{R}^{N} $
\[ 
\sum_{i,j=1}^{N} a_{i,j}(x,t) \xi_i \xi_{j} \geq \mu |\xi|^{2}, \quad (x,t) \in \mathbb{R}^{N} \times [0,T].
 \]

\item[{\bf B2)}] The function $H$ is H\"{o}lder continuous in compact subsets of $\mathbb{R}^{2 N+1}\times[0,T]$. Moreover, there exist $K>0$ and the set $\{K_{m,n}>0:\, m,n \in \mathbb{N}\}$ such that for all $x,\bar{x},p,\bar{p} \in \mathbb{R}^{N}$, $u, \bar{u} \in \mathbb{R}$, $t \in [0,T]$: 
\begin{align}
&|H(0,0,x,t)| \leq K,  \label{b21}\\
&H(0,u,x,t)-H(0,\bar{u},x,t) \leq K (u-\bar{u}) \; \text{  if  } \; u > \bar{u}, \label{b22} \\
&|H(p,u,x,t)-H(p,\bar{u},x,t)| \leq K_{m,n} |u-\bar{u}| \; \text{ if } \; |u|,|\bar{u}| \leq m, \; |x| \leq n, \label{b23} \\
&|H(0,u,x,t)| \leq K_{m,n},  \; \text{  if  } \; |u| \leq m, \; |x| \leq n, \label{boundu} \\
&|H(p,u,x,t)-H(\bar{p},u,x,t)| \leq K_{m,n}  |p-\bar{p}| \; \text{  if  } \; |u| \leq m, \; |x| \leq n. \label{lipschitz}
\end{align}
\item[{\bf B3)}] The function $\beta$ is bounded and Lipschitz continuous on compact subsets of $\mathbb{R}^{N}$.
\end{itemize}
\end{as}

\begin{thm} \label{main}
Under Assumption \ref{as2} there exists a bounded solution $u \in \mathcal{C}^{2,1}(\mathbb{R}^{N}\times[0,T)) \cap \mathcal{C}(\mathbb{R}^{N}\times[0,T])$, \eqref{pierwsze}.
\end{thm}

\begin{proof} Note that for $\varepsilon >0$ we can define the function $a$ and $H$ also for $t \in [-\varepsilon, T]$
 by the formula
\begin{align*}
&a(x,t):=a(x,0) \\  &H(p,u,x,t):=H(p,u,x,0), \quad  t \in [-\varepsilon,0),\; (p,u,x) \in \mathbb{R}^{2N+1}.
\end{align*}
It is useful to notice that $H(p_{N},p_{N-1},\ldots,p_{1},u,x,t)$ can be written as
\begin{align} \label{rewrite}
H(p_{N},p_{N-1},\ldots,p_{1},u,x,t)=&\sum_{i=1}^{N} \dfrac{[H^{i}(p_{i},u,x,t)- H^{i-1}(p_{i-1},u,x,t) ]}{p_{i}} p_{i} \\&+ \frac{[H(0,u,x,t)-H(0,0,x,t)]}{u} u + H(0,0,x,t), \notag
\end{align}
where $H^{i}(p_{i},u,x,t):=H(0,\ldots,0,p_{i},\ldots,p_{2},p_{1},u,x,t)$.
Consider now a new Hamiltonian of the form
\[
H_{k,m,l}(p,u,x,t):=\xi_{k}^{1}(x)\xi_{m}^{2}(u)\xi_{l}^{3}(p) H(p,u,x,t),\quad  k,m,l \in \mathbb{N},
\]
where 
\[
\xi_{k}^{1}(x):=\begin{cases}
1 \quad  & \text{if} \quad |x| \leq k,\\
 \left(2-\frac{|x|}{k}\right), & \text{if} \quad k \leq |x| \leq 2k,\\
  0     &\text{if} \quad  |x| \geq 2k,  \end{cases}
  \]
  \[
\xi_{m}^{2}(u):=\begin{cases}
1 \quad  & \text{if} \quad |u| \leq m,\\
 \left(2-\frac{|u|}{m}\right), & \text{if} \quad m \leq |u| \leq 2m,\\
  0     &\text{if} \quad  |u| \geq 2m,  \end{cases}
  \]
  \[
\xi_{l}^{3}(p):=\begin{cases}
1 \quad  & \text{if} \quad |p| \leq l,\\
 \left(2-\frac{|p|}{l}\right), & \text{if} \quad l \leq |p| \leq 2l,\\
  0     &\text{if} \quad  |u| \geq 2l.  \end{cases}
  \]
We can notice that for a fixed compact  set $B \subset \mathbb{R}^{2N+1} \times [-\varepsilon,T]$ there exists a collection of sufficiently large indices such that
\[H_{k,m,l}(p.u,x,t)=H(p,u,x,t), \quad (p,u,x,t) \in B.
\]
Moreover, for fixed $k,m,l \in \mathbb{N}$ there exists $L(k,m,l)>0$ such that for all
$(p,u,x,t),\; (\bar{p},\bar{u},x,t) \in \mathbb{R}^{2N+1} \times [0,T]$ we have
\begin{equation*}
\begin{aligned}
&|H_{k,m,l}(p,u,x,t)| \leq L(k,l,m)\bigl(1+|p|+|u|\bigr), \\
&|H_{k,m,l}(p,u,x,t)-H_{k,m,l}(\bar{p},\bar{u},x,t)| \leq L(k,l,m) \bigl(|u-\bar{u}| + |p - \bar{p}|\bigr). 
\end{aligned}
\end{equation*}
Therefore,  Theorem \ref{thefirst} can be used within the Hamiltonian $H_{k,m,l}$.
Suppose that $\sigma$ is the unique positive definite square root of $a$. By Friedman \cite[Chapter 6, Lemma 1.1]{friedman2}, $\sigma$ is Lipschitz continuous on compact subsets of $\mathbb{R}^{N} \times [0,T]$. Define
\[
\sigma_{k}(x,t):=\begin{cases}
\sigma(x,t) \quad  & \text{if} \quad |x| \leq k,\\
 \sigma(\frac{kx}{|x|},t), & \text{if} \quad  |x| > k,  \end{cases}, \quad a_{k}:=\sigma_{k} \sigma_{k}^{T}, \quad  \beta_{k}(x):=\xi_{k}^{1}(x) \beta(x).
\]
This yields that there exists $u_{k,m,l} \in \mathcal{C}^{2,1}(\mathbb{R}^{N} \times (-\varepsilon,T)) \cap \mathcal{C}(\mathbb{R}^{N} \times (-\varepsilon,T])$, a bounded solution to 
\begin{align*} 
&u_{t}+\tfrac12\Tr(a_{k}(x,t)D_{x}^{2}u)+ H_{k,m,l}(D_{x}u,u,x,t)=0, &\quad (x,t) \in \mathbb{R}^{N} \times (-\varepsilon,T), \\
&u(x,T)=\beta_{k}(x), &x \in \mathbb{R}^{N}.
\end{align*}
Our reasoning here is based on Arzela - Ascoli's Lemma, so we need to prove some bounds for derivatives of  $u_{k,m,l}$.
Taking advantage of \eqref{rewrite} we can find Borel measurable functions $b_{k,m,l}$, $h_{k,m,l}$ and $f_{k,m,l}$  such that the function $u_{k,m,l}$ is a solution to 
\[
u_{t} + \tfrac12\Tr(a_{k}(x,t)D^{2}_{x}u) + b_{k,m,l}(x,t)D_{x} u + h_{k,m,l}(x,t)u+ f_{k,m,l}(x,t)=0
\]
with the terminal condition $u(x,T)=\beta_{k}(x)$.
Namely, let 
\[f_{k,m,l}(x,t):=H_{k,m,l}(0,0,x,t),\] 

\begin{multline*}
b^{i}_{k,m,l}(x,t) \\:= 
\dfrac{[H^{i}_{k,m,l}(u_{x_{i}}(x,t),u(x,t),x,t)- H^{i-1}_{k,m,l}(u_{x_{i-1}}(x,t),u(x,t),x,t) ]}{ u_{x_{i}}(x,t)},
\end{multline*}
(if   $u_{x_{i}}(x,t) \neq 0$ and $0$ otherwise),
\[
h_{k,m,l}(x,t):=
\begin{cases}
\frac{[H_{k,m,l}(0,u(x,t),x,t)-H_{k,m,l}(0,0,x,t)]}{u(x,t)}, \quad &u(x,t) \neq 0\\
0 , \quad &u(x,t)=0.
\end{cases}
\]
Conditions \eqref{b21} and \eqref{b22} imply
\[
  h_{k,m,l}(x,t)\leq K, \quad |f_{k,m,l}(x,t)|\leq K,
\]
for all $k,m,l \in \mathbb{N}$,  $(x,t) \in \mathbb{R}^{N} \times (-\varepsilon,T]$.
We can now use the standard Feynman - Kac type theorem to obtain stochastic representation of the form:
\begin{multline*}
u_{k,m,l}(x,t) \\ =  \mathbb{E}_{x,t} \left[\int_{t}^{T}e^{\int_{t}^{s}h_{k,m,l}(X_{l},l)\, dl} f_{k,m,l}(X_{s},s) ds  + e^{\int_{t}^{T}h_{k,m,l}(X_{l},l)\, dl}\beta_{k}(X_{T})\right],
\end{multline*}
 where  $dX_{t}=b_{k,m,l}(X_{t},t)dt + \sigma_{k}(X_{t},t)dW_{t}$, $\sigma_{k} \sigma_{k}^{T}=a_{k}$.
The existence of the  strong solution to this stochastic differential equation was proved by Veretennikov \cite{veretennikov}. 
Since functions $\beta$, $f_{k,m,l}$ are bounded and $h_{k,m,l}$ is bounded above, then there exists $m^{*}>0$ independent of $k$ and $m$ such that
\[
|u_{k,m,l}(x,t)| \leq m^{*}. 
\]
This indicates that $u_{k,l}(x,t):=u_{k,m^{*},l}(x,t)$ is a solution to
\begin{equation*} 
 \begin{cases}
u_{t}+\tfrac12\Tr(a_{k}(x,t)D_{x}^{2}u)+ H_{k,l}(D_{x}u,u,x,t)=0, &\quad (x,t) \in \mathbb{R}^{N} \times (-\varepsilon,T), \\
u(x,T)=\beta(x), &\quad x \in \mathbb{R}^{N}, 
\end{cases}
\end{equation*}
where $H_{k,l}(p,u,x,t):= \xi_{k}^{1}(x)  \xi_{l}^{3}(p) H(p,u,x,t)$.
Repeating the procedure described above, we can find Borel measurable functions $b_{k,l}$, $h_{k,l}$ and $f_{k,l}$  such that the function $u_{k,l}$ is a solution to 
\[
u_{t} + \tfrac12\Tr(a_{k}(x,t)D^{2}_{x}u) + b_{k,l}(x,t)D_{x} u + h_{k,l}(x,t)u+ f_{k,l}(x,t)=0
\]
with the terminal condition  $u(x,T)=\beta_{k}(x)$.
We still have 
\[
   h_{k,l}(x,t)\leq K,  \quad |f_{k,l}(x,t)|\leq K 
\]
and 
\[
 |h_{k,l}(x,t)|\leq K_{m^{*},n}, \quad (x,t) \in B_{n} \times [0,T].
\]
We still  need a bound which is independent of  $k,l$. To apply Arzela- Ascoli's Lemma it is sufficient for us to prove such bound for each set $B_{n} \times [-\delta_{n},t_{n}]$, where $B_{n}= \{ x \in \mathbb{R}^{N} |\; |x| \leq n\}$ and $\delta_{n}$ and $t_{n}$ are sequences converging to $\varepsilon$ and $T$ respectively. To get the estimates, we first consider  any function $\varphi$ which satisfies the uniform Lipschitz condition with the constant $L>0$:
\[
|\varphi(z)- \varphi(\bar{z})| \leq L|z-\bar{z}|,   \quad z, \bar{z} \in \mathbb{R}^{N}.
\]
The Lipschitz condition implies the linear growth condition 
\[
|\varphi(z)| \leq L|z| + |\varphi(0)|, \quad z \in \mathbb{R}^{N}.
\]
Next, we need to estimate $|\xi_{l}^{3}(z)\varphi(z)- \xi_{l}^{3}(\bar{z})\varphi(\bar{z})|$ for $ z, \bar{z} \in \mathbb{R}^{N}$. We can assume that $|z| \leq 2l$ or $|\bar{z}| \leq 2l$. Otherwise $|\xi_{l}^{3}(z)\varphi(z)- \xi_{l}^{3}(\bar{z})\varphi(\bar{z})|=0$.
 Without the loss of generality we can assume that $|\bar{z}| \leq 2l$. We have
 \begin{align*}
|\xi_{l}^{3}(z)\varphi(z)- \xi_{l}^{3}(\bar{z})\varphi(\bar{z})| &\leq |\xi_{l}^{3}(z)||
\varphi(z) - \varphi(\bar{z})|+ |\varphi(\bar{z})||\xi_{l}^{3}(z)- \xi_{l}^{3}(\bar{z})| \\\quad &\leq L|z-\bar{z}| +  \left(2lL + |\varphi(0)|\right)|\xi_{l}^{3}(z)- \xi_{l}^{3}(\bar{z})| \\ &\leq
\left[L+\frac{1}{l}\left(2lL+|\varphi(0)| \right)\right]|z-\bar{z}|, \quad z, \bar{z} \in \mathbb{R}^{N}. 
\end{align*}
Therefore, using additionally \eqref{boundu} and \eqref{lipschitz}, we get
\begin{align*}
&\left|H^{i}_{k,l}(u_{x_{i}}(x,t),u(x,t),x,t)- H^{i-1}_{k,l}(u_{x_{i-1}}(x,t),u(x,t),x,t)\right|\\
&  \leq \left[\frac{1}{l}\left(2lK_{m^{*},n}+|H(0,u,x,t)| \right)+K_{m^{*},n}\right] |u_{x_{i}}(x,t)| \\& \leq  \left[\frac{1}{l}\left(2lK_{m^{*},n}+K_{m^{*},n} \right)+K_{m^{*},n}\right] |u_{x_{i}}(x,t)|, \quad (x,t) \in B_{n} \times [0,T], \; k >n.
\end{align*}
This implies that the coefficient $b_{k,l}$ is uniformly bounded on the set $B_{n} \times[-\delta_{n},t_{n}]$ for sufficiently large $l$.    

So far we have obtained uniform bounds for $b_{k,l}$, $h_{k,l}$, $f_{k,l}$ on $B_{n} \times [-\delta_{n},t_{n}]$. To find bounds for $u_{k,l}$, $(u_{k,l})_{t}$, $D_{x}u_{k,l}$,  $D^{2}_{x} u_{k,l}$ and their H\"{o}lder norms uniformly on every set $B_{n} \times [0,t_{n}]$ we use the following reasoning:
\begin{enumerate}
\item Use Lieberman [13, Th. 7.20, Th. 7.22] to get uniform  bounds for $L^{p}(B_{n}\times [-\delta_{n},t_{n}]) $ norms of $u_{k,l}$, $(u_{k,l})_t$, $D_{x}u_{k,l}$, $D^2_{x}u_{k,l}$. For the more general and more readable result it is worth to see Crandall et al. \cite[Theorem 9.1]{Crandall}.
\item Use Fleming and Rishel \cite[Appendix E, E9]{FlemingRishel} to get uniform classical bound for $u_{k,l}$, $D_{x} u_{k,l}$ and their H\"{o}lder norms on the set $B_{n}\times [-\delta_{n},t_{n}]$. 
\item We use bounds for $u_{k,l}$ and $D_{x} u_{k,l}$ to ensure   that for fixed $n \in \mathbb{N}$ and for sufficiently large indexes $k,l$ we have
\[
H_{k,l}(D_{x}u_{k,l}(x,t),u_{k,l}(x,t),x,t)= H(D_{x}u_{k,l}(x,t),u_{k,l}(x,t),x,t), 
\]
for $(x,t) \in B_{n} \times [-\delta_{n},t_{n}]$. 
\item We can use this fact to obtain the uniform bound on the H\"{o}lder norm on the set $B_{n} \times [-\delta_{n},t_{n}]$ for the family $H_{k,l}(D_{x}u_{k,l}(x,t),u_{k,l}(x,t),x,t)$.
\item We already know that $u_{k,l}$ is a classical solution to the problem
\begin{equation*} 
u_{t}+\tfrac12\Tr(a_{k}(x,t)D_{x}^{2}u)+ H_{k,l}(D_{x}u_{k,l},u_{k,l},x,t)=0, \quad (x,t) \in B_{n} \times [-\delta_{n},t_{n}]. \\ 
\end{equation*}
\item Now, it is sufficient to  apply Fleming and Rishel \cite[Appendix E, E10]{FlemingRishel} (which is in fact due to Ladyzhenskaja et al. \cite[Chapter IV, Theorem 10.1]{Lady}) and get uniform classical bounds for the remaining derivatives and their H\"{o}lder norms. 
\end{enumerate} 
\item The bounds for the derivatives ensure that $u_{k,l}$, $(u_{k,l})_{t}$, $D_{x}u_{k,l}$,  $D^{2}_{x} u_{k,l}$ are uniformly bounded, while bounds for the H\"{o}lder norms ensure equicontinuity of $u_{k,l}$, $(u_{k,l})_{t}$ $D_{x}u_{k,l}$,  $D^{2}_{x} u_{k,l}$ on the set
$B_{n} \times [0,t_{n}]$. Thus,
we can use the Arzela - Ascoli's Lemma on each set $B_{n} \times [0,t_{n}]$  to deduce that for each given sequence $(k_{n},l_{n}, n \in \mathbb{N})$ there exists a subsequence $(k_{n_{\mu}},l_{n_{\mu}}, \mu \in \mathbb{N})$ such that  sequences $(u_{k_{n_{\mu}},l_{n_{\mu}}},\; \mu \in \mathbb{N})$,  $((u_{k_{n_{\mu}},l_{n_{\mu}}})_{t},\; \mu \in \mathbb{N})$, $(D_{x} u_{k_{n_{\mu}},l_{n_{\mu}}},\; \mu \in \mathbb{N})$, $( D^{2}_{x}u_{k_{n_{\mu}},l_{n_{\mu}}},\; \mu \in \mathbb{N})$ are convergent uniformly on
$B_{n} \times [0,t_{n}]$. By the standard diagonal argument, there exists a sequence $(k_{n_{\mu}},l_{n_{\mu}}, \mu \in \mathbb{N})$ such that $(u_{k_{n_{\mu}},l_{n_{\mu}}},\; \mu \in \mathbb{N})$ is convergent locally uniformly together with suitable derivatives to a function $u \in C^{2,1}(\mathbb{R}^{N} \times [0,T))$.

Now, we need only to prove that $u$ is continuous at the boundary $\mathbb{R}^{N} \times \{T\}$. Let us apply the It\^o rule to the function $u_{k,l}$  and the stochastic system  $dX_{t}(k)=\sigma_{k}(X_t(k),t)dW_{t}$, and write
\begin{multline*}
\mathbb{E}_{x,t}^{k,l} u_{k,l}(X_{T \wedge \tau_{k}(x,t)}(k),T \wedge \tau_{k}(x,t)) =u_{k,l}(x,t) \\ + \mathbb{E}_{x,t}^{k,l}\int_{t}^{T \wedge \tau_{k}(x,t)}  \left[ -h_{k,l}(X_{s}(k),s) u_{k,l}(X_{s}(k),s)-f_{k,l}(X_{s}(k),s) \right]ds,
\end{multline*}
%\begin{multline*}
%\mathbb{E}_{x,t} u_{k,l}(X_{T \wedge \tau_{k}(x,t)}(k),T \wedge \tau_{k}(x,t))\\ =u_{k,l}(x,t) + \mathbb{E}_{x,t}\int_{t}^{T \wedge \tau_{k}(x,t)}   H_{k,l}(D_{x}u_{k,l}(X_{s}(k),s),u_{k,l}(X_{s}(k),s),X_{s}(k),s) ds,
%\end{multline*}
where $\tau_{k}(x,t)= \inf \{ s\geq t : X_{s}(k) (x,t) \not\in B\}$ for a sufficiently large 
closed ball $B$. The symbol $\mathbb{E}_{x,t}^{k,l}$ is used to denote the expected value under the measure given by the Girsanov transform
\[
\frac{dQ^{k,l}}{dP}:=Z_{x,t,T}^{k,l}
:= e^{\int_{t}^{T \wedge \tau_{k}(x,t)} \sigma_{k}^{-1} b_{k,l}(X_{s}(k),s) dW_{s} -\frac{1}{2}\int_{t}^{T \wedge \tau_{k}(x,t)} |\sigma_{k}^{-1} b_{k,l}(X_{s}(k),s)|^{2} ds} .
\]
 Note that the definition of $\tau$ does not depend on $k$ because there exists $k_{0} \in \mathbb{N}$ such that for all $k \geq k_{0}$ we have $B \subset B_{k}$ and consequently if $k,l \geq k_{0}$ then by Friedman \cite[Theorem 2.1, Section 5] {friedman2} we get $P(\tau_{k}(x,t)=\tau_{l}(x,t))=1$ and $P(\sup_{t \leq s \leq \tau_{k}(x,t)} |X_{s}(k)-X_{s}(l)|=0)=1$. Therefore, we will further omit the variable $k$ in the notation for the process $X$ and the stopping time $\tau(x,t)$. Up to random time $\tau(x,t)$  the process $X$ takes its values in $B$ and the coefficients $b_{k,l}$, $h_{k,l}$, $f_{k,l}$ are uniformly bounded on the set $B \times [0,T]$. 
Let us take any  $(x,t) \in B \times [0,T]$. And let $B$ be a closed ball such that $\bar{x} \in \operatorname{Int} B$, $(x,t) \in \operatorname{Int} B \times [0,T]$. Then,
\begin{multline*}
|u_{k,l}(x,t)-\beta(\bar{x})| \leq \left|\mathbb{E}_{x,t} Z_{x,t,T}^{k,l} u_{k,l}(X_{T \wedge \tau(x,t)},T \wedge \tau(x,t)) - \beta(\bar{x}) \right| \\ +  \mathbb{E}_{x,t}Z_{x,t,T}^{k,l}\int_{t}^{T \wedge \tau(x,t)}  \left| h_{k,l}(X_{s}(k),s) u_{k,l}(X_{s},s) + f_{k,l}(X_{s},s) \right| ds.
\end{multline*}
Furthermore,
\begin{align*}
&\left|\mathbb{E}_{x,t} Z_{x,t,T}^{k,l} u_{k,l}(X_{T \wedge \tau(x,t)},T \wedge \tau(x,t)) - \beta(\bar{x}) \right|\\ & \quad \leq \mathbb{E}_{x,t} Z_{x,t,T}^{k,l} | u_{k,l}(X_{T \wedge \tau(x,t)},T \wedge \tau(x,t)) - \beta(\bar{x})|\\ & \quad \leq \sqrt{\mathbb{E}_{x,t} [Z_{x,t,T}^{k,l}]^{2} } \sqrt{\mathbb{E}_{x,t}| u_{k,l}(X_{T \wedge \tau(x,t)},T \wedge \tau(x,t)) - \beta(\bar{x})|^{2} }.
\end{align*}
The random variable $\left[Z_{x,t,T}^{k,l}\right]^{2}$ can be rewritten as a product of the Girsanov exponent and a uniformly bounded random variable.
In addition, 
\begin{align*}
&\mathbb{E}_{x,t}| u_{k,l}(X_{T \wedge \tau(x,t)},T \wedge \tau(x,t))- \beta(\bar{x})|^{2} \\ & =  \mathbb{E}_{x,t}|\beta (X_{T\wedge \tau(x,t)})- \beta(\bar{x})|^{2} \chi_{\{\sup_{t \leq s \leq T \wedge \tau(x,t)} |X_{s}| < R_{B}\}
} \\ 
&+ \mathbb{E}_{x,t}|u_{k,l} (X_{T \wedge \tau(x,t)},T \wedge \tau(x,t))- \beta(\bar{x})|^{2} \chi_{\{\sup_{t \leq s \leq T \wedge \tau(x,t)}|X_{s}| \geq R_{B}\}}=:I_{1}+I_{2},
\end{align*}
where $R_{B}$ denotes the radius of the ball $B$.
The expression $I_{1}$ is independent of $k,l$ for $k,l \geq k_{0}$ and by the standard diffusion estimates is convergent to $0$ if $(x,t) \to (\bar{x},T)$. The same property holds as well for the second part $I_{2}$  because $|u_{k,l}(x,t)| \leq m^{*}$ and by the martingale inequalities we have
\begin{multline*}
P_{x,t}(\sup_{t \leq s \leq T \wedge \tau(x,t)} |X_{s}| \geq R_{B})  \\ \leq P_{x,t}\left(\sup_{t \leq s \leq T \wedge \tau(x,t)} \left|\int_{t}^{s} \sigma(X_{r}) dW_{r} \right| \geq R_{B} -|x|\right) \\ \quad \leq \frac{1}{ R_{B} -|x|} \mathbb{E}_{x,t} \left| \int_{t}^{T \wedge \tau(x,t)} \sigma(X_{r},r) dW_{r} \right|.
\end{multline*}
Additionally,  by the It\^{o} isometry and the Cauchy -- Schwarz inequality, we have
\begin{multline*}
\mathbb{E}_{x,t} \left| \int_{t}^{T \wedge \tau(x,t)} \sigma(X_{r},r) dW_{r} \right| \\ \leq \left[\mathbb{E}_{x,t} \left| \int_{t}^{T \wedge \tau(x,t)}Tr(\sigma(X_{r},r) \sigma(X_{r},r))dr \right|\right]^{\frac{1}{2}} 
 \to 0, \quad  (\text{if} \; t \to T).
 \end{multline*}
 As a consequence the expression $\left|u_{k,l}(x,t) - \beta(\bar{x})\right|$
admits an estimate which is independent of $k,l$ and converges to $0$ if $(x,t) \to (\bar{x},T)$. Thus, the same property holds true also for $
\left|u(x,t) - \beta(\bar{x})\right| $.
This implies continuity of the function $u$. 
\end{proof}

As it was mentioned in the Introduction, our result can be applied to models with state dependent bounds for the control set. Let us consider the following example describing a variant to the optimal dividend payments problem, which is one of the most important actuarial control problems.
 Let us define the insurer surplus process:
\[
dX_{t} = [\mu-d_{t}] dt+ \sigma dW_{t},
\]
where $\mu, \sigma \in \mathbb{R}$, $\sigma \neq 0$ and $W$ is an one-dimensional Brownian motion. 
The progressively measurable process $d_{t}$ is used to denote the dividend payment intensity. We assume here, that $d_{t}$ can not exceed some fraction of the surplus process and there is no payment at all when the surplus is negative. So, we should always have $0 \leq d_{t} \leq \kappa X_{t}^{+}$. The problem of the insurer is to maximize overall discounted utility of dividend payments, i.e. 
\[
\mathbb{E}_{x,t} \int_{t}^{T} e^{-w(k-t)} f(d_{k}) dk, 
\]
where the function $f$ can be considered as a utility function and $w$ as a discount rate. In the formulation of the problem we can use  as well some penalty function to penalize the objective for allowing the surplus to be negative, but it is not crucial to our analysis. The HJB equation for this problem looks as follows
\[
u_{t} + \frac{1}{2} \sigma^{2} D^{2}_{x} u  + \max_{0 \leq d \leq \kappa x^{+}} \left[(\mu-d) D_{x} u + f(d)\right] - w u =0, \quad u(x,T) =0. 
\]
In this case
\begin{equation} \label{example}
H(p,u,x,t)=\max_{0 \leq d \leq \kappa x^{+}} \left[(\mu-d) p + f(d)-wu\right].
\end{equation} 

More generally we assume that 
\begin{equation} \label{example2}
H(p,u,x,t)=\max_{0 \leq d \leq  m(x,t)} h(p,u,x,t,d).
\end{equation} 

\begin{prop}
Let the function $h$ be Lipschitz continuous on compact subsets of $\mathbb{R}^{3} \times [0,T] \times \mathbb{R}$, satisfy conditions \eqref{b21} and \eqref{b22} uniformly with respect to $d \in \mathbb{R}$ and conditions \eqref{b23}, \eqref{boundu}, \eqref{lipschitz} uniformly with respect to $d \in U$ for all compacts $U \subset \mathbb{R}$ and let the function $m$ be Lipschitz continuous on compact subsets of $\mathbb{R} \times [0,T]$. Then the function $H$, given by \eqref{example2}, satisfies condition  B2.
\end{prop}

\begin{proof}
Almost all conditions contained in ${\bf B2}$ concerning variables $p$ and $u$  are trivial or very easy to prove just using the inequality
\begin{multline*}
|\max_{0 \leq d \leq  m(x,t)} h(p,u,x,t,d) - \max_{0 \leq d \leq m(x,t)} h(\bar{p},\bar{u},x,t,d)| \\ \leq \max_{0 \leq d \leq  m(x,t)}| h(p,u,x,t,d)-h(\bar{p},\bar{u},x,t,d)|.
\end{multline*}

 Local Lipschitz continuity  in $(x,t)$ is much harder to prove. For fixed $(\bar{p},\bar{u}) \in \mathbb{R}^{N+1}$ we have
\begin{multline*}
|\max_{0 \leq d \leq  m(x,t)} h(\bar{p},\bar{u},x,t,d) - \max_{0 \leq d \leq  m(\bar{x},\bar{t})} h(\bar{p},\bar{u},\bar{x},\bar{t},d)| \\ \leq | \max_{0 \leq d \leq  m(x,t)} h(\bar{p},\bar{u},x,t,d) - \max_{0 \leq d \leq  m(x,t)} h(\bar{p},\bar{u},\bar{x},\bar{t},d) | \\+
| \max_{0 \leq d \leq m(x,t)} h(\bar{p},\bar{u},\bar{x},\bar{t},d) - \max_{0 \leq d \leq  m(\bar{x},\bar{t})} h(\bar{p},\bar{u},\bar{x},\bar{t},d)|.
\end{multline*}
The first expression on the right hand side can be estimated using the assumed local Lipschitz continuity i.e. for a given compact set $B \subset \mathbb{R}^{3} \times [0,T]$ there exists $L_{B}>0$ such that for all $(\bar{p},\bar{u},x,t), (\bar{p},\bar{u},\bar{x},\bar{t}) \in B$
\begin{multline}
| \max_{0 \leq d \leq  m(x,t)} h(\bar{p},\bar{u},x,t,d) - \max_{0 \leq d \leq  m(x,t)} h(\bar{p},\bar{u},\bar{x},\bar{t},d) |\\ \leq  \max_{0 \leq d \leq  m(x,t)} | h(\bar{p},\bar{u},x,t,d) - h(\bar{p},\bar{u},\bar{x},\bar{t},d)|  \\ \leq  L_{B} \left( |x-\bar{x}| + |t-\bar{t}| \right). 
\end{multline}

To estimate the second expression we will consider three cases.

{\bf Case I} The maximum  of $\ h(\bar{p},\bar{u},\bar{x},\bar{t},d)$ over $[0 ,m(\bar{x},\bar{t})]$ is attained at some point $ d^{*} \neq m(\bar{x},\bar{t})$, and then by local Lipschitz continuity of $m$ and $h$ we can find sufficiently small neighbourhood of $(\bar{x},\bar{t})$ such that the maximum of $h(\bar{p},\bar{u},\bar{x},\bar{t},d)$ over $0 \leq d \leq m(x,t)$ is still attained at $d^{*}$ . In that case
\[
| \max_{0 \leq d \leq  m(x,t)} h(\bar{p},\bar{u},\bar{x},\bar{t},d) - \max_{0 \leq d \leq  m(\bar{x},\bar{t})} h(\bar{p},\bar{u},\bar{x},\bar{t},d) | =  0.
\]

{\bf Case II} The maximum of $h(\bar{p},\bar{u},\bar{x},\bar{t},d)$ is attained at $d^{*}=m(\bar{x},\bar{t})$ and $m(x,t) <m (\bar{x},\bar{t})$. Then, there still exists 
neighbourhood such that  $\max_{0 \leq d \leq  m(x,t)} h(\bar{p},\bar{u},\bar{x},\bar{t},d)=h(\bar{p},\bar{u},\bar{x},\bar{t},d^{*})$.

{\bf Case III}  The maximum of $h(\bar{p},\bar{u},\bar{x},\bar{t},d)$ is attained at $d^{*}=m(\bar{x},\bar{t})$ and $m(x,t) > m (\bar{x},\bar{t})$. Then, the maximum over $[0,m(x,t)]$ is attained at $\hat{d} \in [m (\bar{x},\bar{t}),m(x,t)]$. In that case we have
\begin{multline*}
|  \max_{0 \leq d \leq  m(\bar{x},\bar{t})} h(\bar{p},\bar{u},\bar{x},\bar{t},d) - \max_{0 \leq d \leq  m(x,t)} h(\bar{p},\bar{u},\bar{x},\bar{t},d) |  \\ = 
| h(\bar{p},\bar{u},\bar{x},\bar{t},d^{*}) -  h(\bar{p},\bar{u},\bar{x},\bar{t},\hat{d}) | .
\end{multline*}
The function $h$ is Lipschitz continuous on compact subsets of $\mathbb{R}^{3} \times [0,T] \times \mathbb{R}$, so for every compact set $B \subset \mathbb{R}^{3} \times [0,T]$ there exists $L>0$ such that for all $(\bar{p},\bar{u},\bar{x},\bar{t}) \in B$
\[
| h(\bar{p},\bar{u},\bar{x},\bar{t},d^{*}) -  h(\bar{p},\bar{u},\bar{x},\bar{t},\hat{d}) | \leq  L |d^{*}-\hat{d}|    \leq 
L|m(\bar{x},\bar{t})-m(x,t)|.
\]
This part can be completed by the assumed local Lipschitz continuity of the function  $m$.

Collecting all inequalities together, we ensure that for any compact set $B \subset \mathbb{R}^{3} \times [0,T]$ there exists a constant $L>0$ such that for any $(\bar{p},\bar{u},\bar{x},\bar{t}) \in B$ there is a small neighbourhood $U_{(\bar{p},\bar{u},\bar{x},\bar{t})}$ such that
\begin{multline} \label{holds}
| \max_{0 \leq d \leq  m(x,t)} h(p,u,x,t,d) - \max_{0 \leq d \leq  m(\bar{x},\bar{t})} h(\bar{p},\bar{u},\bar{x},\bar{t},d) | \\ \leq L|(p,u,x,t)-(\bar{p},\bar{u},\bar{x},\bar{t})|,  
\end{multline}
for all $(p,u,x,t,d) \in U_{(\bar{p},\bar{u},\bar{x},\bar{t})}$. The fact that the constant $L>0$ depends only on the compact set $B$ and not on the particular choice of the point $(\bar{p},\bar{u},\bar{x},\bar{t})$ implies  local Lipschitz continuity of the function $H$. Namely, let $B \subset \mathbb{R}^{3} \times [0,T]$ be a compact and convex set of the form $\{x \in \mathbb{R}^{3}| \; |x| \leq R\} \times [0,T]$. Fix $z=(p,u,x,t),\bar{z}=(\bar{p},\bar{u},\bar{x},\bar{t}) \in B$ and consider the compact set (the line connecting $z$ and $\bar{z}$)
\[
O_{[z,\bar{z}]}=\{z+ \alpha(\bar{z}-z), \alpha \in [0,1]\} \subset B.
\]
For each point $z \in B$ there exists $U_{z}$ (we may assume this is an open ball) such that \eqref{holds} holds. Compactness of the set $O_{[z,\bar{z}]}$ implies that there exist finitely many points $z=z_{1},z_{2},\ldots,z_{n}=\bar{z} \in O_{[z,\bar{z}]} $ (we can order them according to the increasing euclidean distance from the point $z$) such that for every ordered pair $z_{i},z_{i+1}$ we have $ z_{i},z_{i+1} \in \bar{U}_{z_{i}}$ or $ z_{i},z_{i+1} \in \bar{U}_{z_{i+1}}$. In that case we have
\[
|H(z)-H(\bar{z})| \leq \sum_{i=2}^{n}|H(z_{i})-H(z_{i-1})| \leq L \sum_{i=2}^{n}|z_{i}-z_{i-1}| = L|z-\bar{z}|.
\]

%\begin{exa} Our results can be applied for instance to the  semilinear Burgers type Cauchy problem of the form
%\[
%\begin{cases}
%u_{t}+\frac{1}{2} Tr(\lambda(x)D^{2} u) + \mu(x) u D_{x} u =0,  &\quad (x,t) \in \mathbb{R}^{N} \times [0,T),\\
%u(x,T)=\beta(x), &\quad x \in \mathbb{R}^{N}.
%\end{cases}
%\]
%\end{exa}
\end{proof}
\section{Isaacs equation}

Now, our primary concern is to solve the following semilinear equation
\begin{multline}  \label{equationL}  u_{t}+ \tfrac12\Tr( a(x,t)D^{2}_{x} u)  \\ +\max_{\delta \in D} \min_{\eta \in \Gamma} \biggl(i(x,t,\delta,\eta) D_{x}u  +  h(x,t,\delta,\eta) u+ f(x,t,\delta,\eta)\biggr) =0, \\ \quad (x,t) \in \mathbb{R}^{N} \times [0,T)
\end{multline}
with the terminal condition
$u(x,T)=\beta(x)$.
 \begin{as} \label{as3}
 \text{ }
 \begin{itemize}
\item[{\bf C1)}]
The matrix $[a_{i,j}(x,t)]$, $i,j=1,2,\ldots, N$ is symmetric, the coefficients  are Lipschitz continuous on compact subsets in $\mathbb{R}^{N} \times [0,T]$. In addition there exists a constant $\mu>0$ such that
 for any  $\xi \in \mathbb{R}^{N} $
\[ 
\sum_{i,j=1}^{N} a_{i,j}(x,t) \xi_i \xi_{j} \geq \mu |\xi|^{2}, \quad (x,t) \in \mathbb{R}^{N} \times [0,T].
 \]

\item[{\bf C2)}] Functions $f$, $h$, $i$  are continuous, there exists a strictly positive sequence $\{L_{n}\}_{n \in \mathbb{N}}$ such  that for all $\zeta=f,h,i$ and for all  $\delta \in D$, $\eta \in \Gamma$, $(x,t) \in B_{n} \times [0,T]$
\begin{equation*} 
|\zeta(x,t,\delta,\eta)-\zeta(\bar{x},\bar{t},\delta,\eta)|\leq L_n( |x- \bar{x}|+|t-\bar{t}|). \\ 
 \label{lipcond1}
\end{equation*}
\item[{\bf C3)}] The function $\beta$ is uniformly bounded and Lipschitz continuous on compact subsets of $\mathbb{R}^{N}$.
\item[{\bf C4)}] The function $f$ is uniformly bounded and $h$ is bounded 
above.
\end{itemize}
 \end{as}
 
We can now present an immediate consequence of Theorem \ref{main}.
 
\begin{cor} Under Assumption \ref{as3}, there exists a bounded classical solution $u \in \mathcal{C}^{2,1}(\mathbb{R}^{N}\times[0,T)) \cap \mathcal{C}(\mathbb{R}^{N}\times[0,T])$ to \eqref{equationL}. 
\end{cor}

\begin{proof}
For the proof it is sufficient to define
\[
H(p,u,x,t):=\max_{\delta \in D} \min_{\eta \in \Gamma} \Pi(p,u,x,t,\delta,\eta),
\]
where
\[
\Pi(p,u,x,t,\delta,\eta)=i(x,t,\delta,\eta)p +  h(x,t,\delta,\eta) u+ f(x,t,\delta,\eta)
\]
and use the inequality
\[
|H(p,u,x,t) - H(\bar{p},\bar{u},\bar{x},\bar{t}) | \leq \max_{\delta \in D} \max_{\eta \in \Gamma} |\Pi(p,u,x,t,\delta,\eta)-\Pi(\bar{p},\bar{u},\bar{x},\bar{t},\delta,\eta)|.
\]
\end{proof} 
 
In some cases it is possible to extend the above result to the case when functions $f$ and $g$ might be unbounded. We need first the following Lemma:

\begin{lem} \label{lemma}
Assume that $X(n)$ stands for a strong solution to 
\[
dX_{t}=b_{n}(X_{t},t,\omega)dt + \sigma_{n}(X_{t},t,\omega) dW_{t},
\]
where $b_{n}$ and $\sigma_{n}$ are sequences of  functions such that 
\[
b_{n}: \mathbb{R}^{N} \times [0,T] \times \Omega \to \mathbb{R}^{N} , \; \sigma_{n}: \mathbb{R}^{N} \times [0,T] \times \Omega  \to L(\mathbb{R}^{N}, \mathbb{R}^{N})
\]
and there exist $K,M>0$ such that  for all $x \in \mathbb{R}^{N}$,  $n \in \mathbb{N}$ and $\omega \in \Omega$
\[
|b_{n}(x,t,\omega)|\leq K(1+|x|), \; |\sigma_{n}(x,t,\omega)| \leq M.
\]
Then for all $A>0$ there exists a continuous function $\hat{R}$ such that for all $ n \in \mathbb{N}$ and $(x,t) \in \mathbb{R}^{N} \times [0,T]$
\[
\mathbb{E}_{x,t} \sup_{t \leq s \leq T} e^{A |X_{s}(n)|}  \leq \hat{R}(x). 
\]
\end{lem}

\begin{proof} We start with proving some pathwise inequalities which hold almost surely in $\Omega$. 
 If $b_{n}$ satisfies the linear growth condition, then there exists $K>0$ such that for all $k \in [t,T]$
\[
 |X_{k}(n)| \leq|x|+ KT + K\int_{t}^{k} |X_{s}(n)| ds + \sup_{0 \leq k \leq T} \left|\int_{t}^{k} \sigma_{n}(X_{s}(n),s,\omega) dW_{s}\right|.
\]
Therefore,
\[
 |X_{k}(n)| \leq A_{T} + K\int_{t}^{k} |X_{s}(n)| ds, \quad  t \leq  k \leq T,
\] 
where 
\[
A_{T} := \left(|x|+KT+ \sup_{t \leq k \leq T} \left|\int_{t}^{k} \sigma_{n}(X_{s}(n),s,\omega) dW_{s}\right|\right).
\] 
Using the Gronwall inequality we have
\[
|X_{k}(n)| \leq \left(|x|+KT+ \sup_{t \leq k \leq T} \left|\int_{t}^{k} \sigma_{n}(X_{s}(n),s,\omega) dW_{s}\right|\right)e^{KT}.
\]
Therefore, it is sufficient to find a uniform bound for
\[
\mathbb{E}_{x,t} \sup_{t \leq k \leq T} e^{A \left|\int_{t}^{k}\sigma_{n}(X_{s}(n),s,\omega) dW_{s} \right|}.
\] 
Note that  $e^{|Z|} \leq e^{Z}+e^{-Z}$ and
\begin{multline*}
\mathbb{E}_{x,t} \sup_{t \leq k \leq T} e^{A \int_{t}^{k} \sigma(X_{s}(n),s,\omega) dW_{s}} \\=
\mathbb{E}_{x,t} \sup_{t \leq k \leq T} e^{A \int_{t}^{k} \sigma_{n}(X_{s}(n),s,\omega) dW_{s}-\frac{1}{2}A^{2}\int_{t}^{k} Tr(\sigma_{n}(X_{s}(n),s,\omega) \sigma_{n}^{T}(X_{s}(n),s,\omega)ds } \times \\ \times e^{\frac{1}{2}A^{2}\int_{t}^{k} Tr(\sigma_{n}(X_{s}(n),s,\omega) \sigma_{n}^{T}(X_{s}(n),s,\omega))ds }.
\end{multline*}
Since $\sigma_{n}$ is uniformly bounded, the process
\[
e^{\frac{1}{2} A^{2}\int_{t}^{k} Tr(\sigma_{n}(X_{s}(n),s,\omega) \sigma_{n}^{T}(X_{s}(n),s,\omega))ds }
\]
is bounded as well.
Now, we can use  the martingale inequality to ensure the existence of a uniform constant $C_{T}>0$ such that
\[
\mathbb{E}_{x,t} \sup_{t \leq k \leq T} M_{k}^{n}  \leq C_{T} \sqrt{\mathbb{E}_{x,t} \left[M_{T}^{n}\right]^{2}}, 
\]
where
\[
M_{t}^{n} := e^{A \int_{t}^{k} \sigma_{n}(X_{s}(n),s,\omega) dW_{s}-\frac{1}{2}A^{2}\int_{t}^{k} Tr(\sigma_{n}(X_{s}(n),s,\omega) \sigma_{n}^{T}(X_{s}(n),s,\omega))ds } .
\]
The conclusion follows from the fact that $\left[M_{T}^{n}\right]^{2}$ can be rewritten in the form 
\[
[M_{T}^{n}]^{2} =G_{T}^{n} N_{T}^{n},
\]
where the random variable $G_{T}^{n} $ is used to change  the measure and  the family $N_{T}^{n}$ is  uniformly bounded.
\end{proof}

\begin{as} \label{as4}
  \text{ }
\begin{itemize}
\item[{\bf D1)}]
The matrix $a$ is symmetric, $a=\sigma \sigma^{T}$, the coefficients $\sigma_{i,j}(x,t)$, $i,j=1,\ldots, N$ are Lipschitz continuous on compact subsets in $\mathbb{R}^{N} \times [0,T]$. In addition there exists a constant $\mu>0$ such that for any  $\xi \in \mathbb{R}^{N} $
\[ 
\sum_{i,j=1}^{N} a_{i,j}(x,t) \xi_i \xi_{j} \geq \mu |\xi|^{2}, \quad (x,t) \in \mathbb{R}^{N} \times [0,T].
 \]

\item[{\bf D2)}] Functions $f$, $h$, $i$  are continuous and there exists a strictly positive sequence $\{L_{n}\}_{n \in \mathbb{N}}$ such  that for all $\zeta=f,h,i$ and for all $\delta \in D$, $\eta \in \Gamma$, $(x,t) \in B_{n} \times [0,T]$
\begin{equation*} 
|\zeta(x,t,\delta,\eta)-\zeta(\bar{x},\bar{t},\delta,\eta)|\leq L_n( |x- \bar{x}|+|t-\bar{t}|). \\  
\end{equation*}

\item[{\bf D3)}] The function $\beta$ is Lipschitz continuous on compact subsets of $\mathbb{R}^{N}$.

\item  [{\bf D4)}] There exist $A,B>0$ such that for all $\delta \in \mathcal{D}$, $\eta \in \mathcal{N}$, $(x,t) \in \mathbb{R}^{N} \times [0,T]$
\begin{gather*}
|f(x,t,\delta,\eta)| +  |\beta(x)| \leq  B e^{A|x|}, \;
|\sigma(x,t)| \leq B \\
|h(x,t,\delta,\eta)|+|i(x,t,\delta,\eta)|  \leq B(1+|x|)
\end{gather*}
or for all $\delta \in \mathcal{D}$, $\eta \in \mathcal{N}$, $(x,t) \in \mathbb{R}^{N} \times [0,T]$
\begin{gather*}
|f(x,t,\delta,\eta)| +  |\beta(x)| \leq  B e^{A|x|}, \;
|\sigma(x,t)| \leq B ,\\
 |i(x,t,\delta,\eta)|  \leq B(1+|x|), \quad h(x,t,\delta,\eta) \leq B.
\end{gather*}
\end{itemize}
\end{as}
\begin{thm} \label{last} Under Assumption \ref{as4}, there exists a classical solution $u \in \mathcal{C}^{2,1}(\mathbb{R}^{N}\times[0,T)) \cap \mathcal{C}(\mathbb{R}^{N}\times[0,T])$ to \eqref{equationL}.
\end{thm}
\begin{proof} Define

\[
\sigma_{n}(x,t):=\begin{cases}
\sigma(x,t) \quad  & \text{if} \quad |x| \leq n,\\
 \sigma(\frac{nx}{|x|},t), & \text{if} \quad  |x| \geq n,  \end{cases}\quad  a_{n}:=\sigma_{n} \sigma_{n}^{T}, \quad \beta_{n}(x):= \zeta_{n}(x) \beta(x),
\]

\[
i_{n}(x,t,\delta,\eta):= \zeta_{n} (x)i(x,t,\delta,\eta), \quad
 f_{n}(x,t,\delta,\eta):= \zeta_{n} (x)f(x,t,\delta,\eta),
 \]
\[
 h_{n}(x,t,\delta,\eta):= \zeta_{n} (x)h(x,t,\delta,\eta) \]
 \[ \text{ or }   h_{n}(x,t,\delta,\eta):= h(x,\delta,\eta)\; (\text{if } \; h(x,t,\delta,\eta) \leq B),
\] 
 where
\[
\zeta_{n}(z):=\begin{cases}
1 \quad  & \text{if} \quad |z| \leq n,\\
 \left(2-\frac{|z|}{n}\right), & \text{if} \quad n \leq |z| \leq 2n,\\
  0        &\text{if} \quad  |z| \geq 2n. \end{cases}
\]
Functions $a_{n}$, $f_{n}$, $i_{n}$, $h_{n}$, $\beta_{n}$ are bounded (or bounded above in the case of the function $h_{n}$) and we still have 
\begin{gather} 
|h_{n}(x,t,\delta,\eta)| +|i_{n}(x,t,\delta,\eta)| \leq B(1+|x|) \;\text{ or } \; h_{n}(x,t,\delta,\eta) \leq B,  \label{1}\\
|f_{n}(x,t,\delta,\eta)|   + |\beta_{n}(x)| \leq  B e^{A|x|}, \\
|\sigma_{n}(x,t)| \label{2} \leq B.
\end{gather}
Let $u_{n}$ denote any classical solution to the equation 
\begin{multline} \label{eq}   u_{t}+ \tfrac12\Tr( a_{n}(x,t) D_{x}^{2}u) \\  +\max_{\delta \in D} \min_{\eta \in \Gamma} \biggl(i_{n}(x,t,\delta,\eta) D_{x} u  +  h_{n}(x,t,\delta,\eta) u+ f_{n}(x,t,\delta,\eta)\biggr) =0, \\ \quad (x,t) \in \mathbb{R}^{N} \times [0,T),
\end{multline}
with the terminal condition $u(x,T)=\beta_{n}(x,T)$.
Using measurable selection theorems to $\min$ and $\max$ in \eqref{eq}, we can find 
Borel measurable coefficients $i_{n}^{*}, f_{n}^{*}, h_{n}^{*}$ such that $u_{n}$ is a solution to
\begin{equation} \label{n:eq}
u_{t}+ \tfrac12\Tr( a_{n}(x,t) D_{x}^{2}u_{n}) +  \biggl(i_{n}^{*}(x,t) D_{x} u_{n}  +  h_{n}^{*}(x,t) u_{n}+ f_{n}^{*}(x,t)\biggr) =0.
\end{equation}
For $u_{n}$ we have the following stochastic representation 
\[
u_{n}(x,t)=\mathbb{E}_{x,t}\left( \int_{t}^{T} e^{\int_{t}^{s} h_{n}^{*}(X_{k},k) dk}f_{n}^{*}(X_{s},s) ds + e^{\int_{t}^{T} h_{n}^{*}(X_{k},k) dk}\beta_{n}(X_{T})\right).
\]
Since we have \eqref{1} -- \eqref{2}, we can use  Lemma \ref{lemma} to guarantee that there exists a continuous function $R(x)$ such that for all $(x,t) \in \mathbb{R}^{N} \times [0,T]$
\begin{equation*} \label{uniform:n}
\sup_{n} \mathbb{E}_{x,t}\left( \int_{t}^{T} e^{\int_{t}^{s} h_{n}^{*}(X_{k},k) dk}|f_{n}^{*}(X_{s},s) |ds + e^{\int_{t}^{T} h_{n}^{*}(X_{k},k) dk}|\beta_{n}(X_{T})|\right) \leq R(x).
\end{equation*}
and consequently  $\sup_{n}|u_{n}(x,t)| \leq R(x)$. Now we can use the reasoning from the proof of Theorem \ref{main}.

%\begin{enumerate}
%\item Use Lieberman \cite[Th. 7.20, 7.22]{Lieberman} to get uniform (on the  ball $B_{k}$)bounds  for $W^{2,p}$ Sobolev norms of $u_{n}$, $(u_{n})_{t}$, $D_{x}u_{n}$,  $D^{2}_{x}u_{n}$. 
%\item Use Fleming and Rishel \cite[Appendix E, E9]{FlemingRishel} to get the uniform classical bound for $D_{x} u_{n}$ and its H\"{o}lder norm. 
%\item Rewrite equation \eqref{n:eq} into
%\[
%u_{t}+ \frac{1}{2} Tr( a_{n}(x,t) D_{x}^{2}u) +  G_{n}(D_{x}u,u,x,t) =0
%\] 
%and use the fact we have uniform bounds for H\"{o}lder norms of \\ $G_{n}(D_{x}u_{n},u_{n},x,t)$ to apply Fleming and Rishel \cite[Appendix E, E9]{FlemingRishel}  and get uniform classical bounds for the rest suitable derivatives and their H\"{o}lder norms. 
%\end{enumerate} 
%Above estimates are sufficient to apply the Arzell-Ascolli Lemma on each $B_{k}$ and get the existence theorem for our equation. Continuity on the boundary $\mathbb{R}^{N} \times \{T \}$ can be proved as in the proof of Theorem \ref{main}.
%\end{proof}
%
\end{proof}

{\bf Acknowledgments}. I would like to gratefully thank referees for a
careful reading of the manuscript and a list of helpful recommendations.

% ------------------------------------------------------------------------
\end{document}